\documentclass[10pt]{amsart}
\usepackage{latexsym}
\usepackage{amsmath}
\usepackage{amsfonts}
\usepackage{amsthm}
\usepackage{amssymb}
\usepackage[T1]{fontenc}
\begin{document}

\title[functional inequality]{On a functional inequality stemming from the Gr\"uss-type inequalities}

\author{W\l{}odzimierz~Fechner}
\address{Institute of Mathematics, Lodz University of Technology, ul. W\'olcza\'nska 215, 90-924 \L\'od\'z, Poland}
\email{wlodzimierz.fechner@p.lodz.pl}

\newtheorem{thm}[]{Theorem}
\newtheorem{cor}[]{Corollary}
\newtheorem{lem}[]{Lemma}
\newtheorem{prop}[]{Proposition}
\theoremstyle{remark}
\newtheorem{rem}[]{Remark}
\newtheorem{ex}[]{Example}
\newcommand{\N}{\mathbb{N}}
\newcommand{\R}{\mathbb{R}}
\newcommand{\C}{\mathbb{C}}
\newcommand{\K}{\mathbb{K}}
\newcommand{\f}{\varphi}
\newcommand{\g}{\psi}
\renewcommand{\(}{\left(} \renewcommand{\)}{\right)}
\renewcommand{\[}{\left[} \renewcommand{\]}{\right]}


\begin{abstract}
We study functional inequality of the form
$$|T(f,h)-T(f,g)T(g,h)| \leq F(f,g)F(g,h) -F(f,h)$$
where $T$ is a complex-valued functional and $F$ is a  real-valued map. Motivation for our studies comes from some generalizations of  Gr\"uss inequality. 
\end{abstract}

\maketitle

\section{Gr\"uss-type inequalities}

Let $[a, b]\subset \R$ be a non-degenerated interval and $f, g \in L([a,b])$ be Lebesgue integrable functions. Assume that there exist constants $m_f, m_g, M_f, M_g \in \R$ such that
\begin{equation}
\label{con}
m_f \leq f \leq M_f \quad \textrm{and} \quad m_g \leq g \leq M_g \qquad \textrm{a.e. on }[a,b].
\end{equation} 
Let us denote integral mean of a function $f$ by $I(f)$:
$$ I(f) = \frac{1}{b-a} \int\limits_a^b f(x) dx , \quad f \in L([a,b]).$$

The celebrated Gr\"uss inequality dates as early as 1935 and says that
\begin{equation}
\label{G}
\left| I(fg) -  I(f)I(g)     \right| \leq \frac14(M_f-m_f)(M_g-m_g).
\end{equation}
Nowadays several generalizations and refinements are known. For a comprehensive study of the topic and the list of references we refer the reader to a monograph by S.S. Dragomir \cite{D} and to recent papers by Z. Otachel \cite{Ot, Ot2}.

A Gr\"uss-type inequality is an inequality which provides an upper bound for expression of the form
$$\left|\left\langle f,h \right\rangle - \left\langle f,g \right\rangle \left\langle g,h \right\rangle \right|,$$
sometimes under additional assumption $\|g\|=1$, or 
$$\left|\left\langle f,h \right\rangle - \frac{1}{\|g\|^2}\left\langle f,g \right\rangle \left\langle g,h \right\rangle\right|,$$
where $f, g, h$ belong to an inner product space (see \cite[Theorems 15, 16, 17]{D}). Cases with inner product replaced by another functional are studied as well, including discrete versions of the original inequality. 

In paper \cite{ja} we dealt with functionals which satisfy
\begin{equation}
|T(f,h)-T(f,g)T(g,h)|\leq c \quad \textrm{for all } f, g, h
\label{pams}
\end{equation}
with some constant $c\geq 0$. Research of \cite{ja} was motivated by Richard's inequality, which says that particular solutions of \eqref{pams} are functionals of the form
\begin{equation}
T(f,g) = \frac{\left\langle f,g \right\rangle}{\|f\|\cdot \|g\|}.
\label{T}
\end{equation}
In the main result of \cite{ja} we proved that every unbounded real-valued solution of \eqref{pams} is a solution of Sincov equation
\begin{equation}
T(f,h)=T(f,g)T(g,h) \quad \textrm{for all } f, g, h.
\label{S}
\end{equation}
Consequently, $T$ has the representation 
$$T(f,g) = \frac{\Phi(f)}{\Phi(g)} \quad \textrm{for all } f, g$$
with some arbitrary never-vanishing real-valued mapping $\Phi$ (see D. Gronau \cite[Theorem]{G}). Therefore, from the mail result of \cite{ja} it follows that no generalization of Richard's inequality such that the functional $T$ is replaced by an unbounded one  is possible.

The purpose of the present work is to deal with a more general problem, with the right-hand side of \eqref{pams} modified so that it can depend on $f, g$ and $h$ (like in the case of Gr\"uss inequality and its generalizations). Moreover, we will study additive analogue of the original problem. In some proofs we will utilize our results from \cite{ja2} on Sincov's inequalities on topological spaces.

\section{Multiplicative functional inequality}

Assume that $X$ is a nonempty set and $T$ and $F$ are scalar mappings acting on the product $X \times X$. We will study the following functional inequality:
\begin{equation}
|T(f,h)-T(f,g)T(g,h)|\leq F(f,g)F(g,h)-F(f,h).
\label{main}
\end{equation}

Clearly, if $c\geq 0$ is a fixed constant and one takes as $F$ constant map equal to $1/2(1+ \sqrt{1+4c})$, then \eqref{main} reduces to \eqref{pams}. 

Present study is meant to fall in line with research initiated by the notion of delta-convexity by L. Vesel\'y, L. Zaj\'i\v{c}ek \cite{VZ} and then continued by several authors. Dissertation \cite{VZ} is mainly devoted to inequality
$$\left\|  F\( \frac{x+y}{2} \) - \frac{F(x)+F(y)}{2}       \right\| \leq  \frac{f(x)+f(y)}{2} - f\( \frac{x+y}{2} \) .$$
Continuous solutions $F$ of this inequality are called delta-convex. One of theorems of \cite{VZ} says that every delta-convex function can be written as a difference of two convex functionals.
This result provided a motivation for a study of several related problems. In particular, R. Ger \cite{G} studied inequality
$$\left\|  F\( {x+y} \) - F(x)F(y)   \right\| \leq  f(x)f(y) - f\( {x+y} \),$$
using the term delta-exponential map for $F$. More recently, A. Olbry\'s dealt with delta $(s,t)$-convex mappings \cite{O1}, delta Schur-convex mappings \cite{O2} and delta-subadditive and delta-superadditive mappings \cite{O3}.

We begin with an observation that inequality \eqref{main} behaves in a way similar to other inequalities motivated by the notion of delta-convexity. Proposition \ref{p1} below is an analogue to \cite[Proposition 1]{G}.

\begin{prop}\label{p1}
Assume that $X$ is a nonempty set and $T\colon X \times X \to [0, + \infty) $ and $F\colon X \times X \to [0, + \infty)$ satisfy inequality \eqref{main} for all $f, g, h \in X$. If we denote $H=T+F$, then
\begin{equation}
\label{multI}
H(f,h)\leq H(f,g)H(g,h), \quad f, g, h \in X.
\end{equation}
\end{prop}
\begin{proof}
Fix $f, g, h \in X$; we have by \eqref{main}
\begin{align*}
H(f,h) &= T(f,h) + F(f,h) \leq T(f,g)T(g,h) + F(f,g)F(g,h) \\&= [H(f,g) - F(f,g) ][H(g,h) - F(g,h)] + F(f,g)F(g,h)\\&= H(f,g)H(g,h) - H(f,g)F(g,h) - F(f,g)H(g,h) + 2F(f,g)F(g,h)\\&= H(f,g)H(g,h) - T(f,g)F(g,h) - F(f,g)T(g,h)  \leq H(f,g)H(g,h).
\end{align*}
\end{proof}

One can apply a result from \cite{ja2} in order to exclude cases when $F$ attains zero. Indeed, by \cite[Proposition 2]{ja2}, if $F$ has a zero, then $F=0$ on $X \times X$  and, consequently, $T$ solves Sincov equation \eqref{S}. Therefore, from now on we will restrict ourselves to the case when $F$ is positive.

In our next result we describe solutions of \eqref{main} which satisfy an additional assumption.

\begin{thm}
Assume that $X$ is a nonempty set and $T\colon X \times X \to \C $ and $F\colon X \times X \to (0, + \infty)$ satisfy inequality \eqref{main} for all $f, g, h \in X$. If there exists some $f, g \in X$ such that the map 
\begin{equation}
X \ni k  \to \frac{T(g,k)}{F(f,k)}
\label{map}
\end{equation}
 is unbounded, then
$T$ solves Sincov equation \eqref{S} for all $f, g, h \in X$. 
\end{thm}
\begin{proof}
We begin with strengthening the unboundedness assumption. We claim that for every $f, g \in X$ the map \eqref{map} is unbounded. Note that \eqref{main} implies immediately the inequality
\begin{equation}
\label{F}
F(f,h)\leq F(f,g)F(g,h), \quad f, g, h \in X.
\end{equation}
By \eqref{F} we have
$$\frac{|T(g,k)|}{F(f,k)}\leq  \frac{|T(g,k)|}{F(h,k)} F(h,f), \quad f, g, h, k \in X.$$
One can see that if the left-hand side is unbounded as a variable of $k$ with $f,g$ kept fixed , then the fraction on the right-hand side is unbounded with every $h$. Next, from \eqref{main} we get
$$\left| T(g,k) - T(g,f)T(f,k) \right|\leq F(g,f)F(f,k)-F(g,k)\leq F(g,f)F(f,k), \quad f, g, k \in X;$$
therefore
$$\left| \frac{T(g,k)}{F(f,k)} - T(g,f)\frac{T(f,k)}{F(f,k)} \right|\leq F(g,f)\quad f, g, k \in X.$$
Thus, if for fixed $f,g$ the first fraction is unbounded, then so is the second one. On joining both observations our first claim follows.

Define an auxiliary functional $\Gamma\colon X \times X \to \R$ as
$$\Gamma(f, g) = F(f,g)F(g,f) - 1, \quad f,g \in X.$$ 
Using inequality \eqref{F} for arbitrary  $f, g, h \in X$  we get
$$F(g, h) \leq F(g, f)F(f, h),$$
so, by positivity of $F$
$$\frac{F(f,g)F(g,h)}{F(f,h)}\leq F(f,g)F(g,f) = \Gamma(f, g)+1.$$
Thus
\begin{equation}\label{gamma}
F(f,g)F(g,h) - F(f,h) \leq \Gamma(f, g)F(f,h), \quad f, g, h \in X.
\end{equation}

Take arbitrary $f, g, h, k \in X$. By \eqref{main} and \eqref{gamma} we obtain
$$|T(f,k) - T(f,h)T(h,k)|\leq   F(f,h)F(h,k)-F(f,k) \leq \Gamma(f, h) F(f,k)$$
and
$$|T(f,k) - T(f,g)T(g,k)|\leq F(f,g)F(g,k)-F(f,k)\leq \Gamma(f, g) F(f,k),$$
therefore
\begin{equation}\label{1}
|  T(f,h)T(h,k)-T(f,g)T(g,k)|\leq [\Gamma(f, g) + \Gamma(f, h)] F(f,k). 
\end{equation}
Using \eqref{1}, \eqref{main} and then \eqref{gamma}  we arrive at
\begin{align*}
| T(h,k) |&\cdot |T(f,h)-T(f,g)T(g,h)| = |T(f,h)T(h,k)-T(f,g)T(g,h)T(h,k)|\\&\leq [\Gamma(f, g) + \Gamma(f, h)] F(f,k)+ | T(f,g)T(g,k)- T(f,g)T(g,h)T(h,k)|\\&\leq  [\Gamma(f, g) + \Gamma(f, h)] F(f,k)+ | T(f,g)|[F(g,h)F(h,k) - F(g,k)] \\&\leq [\Gamma(f, g) + \Gamma(f, h)] F(f,k)+ | T(f,g)| \Gamma(g, h)\cdot F(g,k).
\end{align*}
From this we derive
$$
|T(f,h)-T(f,g)T(g,h)| \leq  [\Gamma(f, g) + \Gamma(f, h)]\frac{F(f,k)}{| T(h,k) |}   + |T(f,g)| \Gamma(g, h)\frac{ F(g,k)}{| T(h,k) |}  .
$$
Utilizing our preliminary observation we see that the two fractions can be arbitrary small while $f, g, h$ are kept fixed. Thus we derive that the left-hand side is equal to zero.
\end{proof}

\section{Additive functional inequality}

In this section we provide analogues to the foregoing result in additive case. It turns out that the situation is much easier and more symmetric.  Assume that $X$ is a nonempty set and $T$ and $F$ are scalar mappings acting on $X \times X$. We will deal with the  functional inequality:
\begin{equation}
|S(f,h)-S(f,g) -S(g,h)|\leq G(f,g)+G(g,h)-G(f,h).
\label{add}
\end{equation}
The following result is straightforward to verify.

\begin{prop}\label{p2}
Assume that $X$ is a nonempty set and $S\colon X \times X \to \R$ and $G\colon X \times X \to \R$ satisfy inequality \eqref{add} for all $f, g, h \in X$. If we denote $H_1=G+S$ and $H_2=G-S$, then
\begin{equation}
\label{addI}
H_i(f,h)\leq H_i(f,g)+ H_i(g,h), \quad f, g, h \in X, \, i=1, 2,
\end{equation}
 $S=\frac12H_1 - \frac12H_2$ and $G=\frac12H_1 + \frac12H_2$.
\end{prop}
In particular, every mapping $S$ satisfying \eqref{add} can be written as a difference of two solutions of inequality \eqref{addI}. The converse implication is also true and the proof is omitted as straightforward.
\begin{prop}\label{p3}
Assume that $X$ is a nonempty set and $H_i\colon X \times X \to \R$ for $i=1, 2$ satisfy inequality \eqref{addI} for all $f, g, h \in X$. 
If we denote $S=H_1+H_2$ and $G =H_1-H_2$, then $S$ and $G$ satisfy inequality \eqref{add}.
\end{prop}

In \cite{ja2} we described continuous solutions of \eqref{addI} in some classes of spaces. Let us quote \cite[Corollary 5]{ja2}.

For arbitrary function $H\colon X \times X \to \R$ define
$$\mathcal{H}(H) = \left\{  \f\colon X \to \R : \forall_{f, g \in X} \, \f(f)-\f(g) \leq H(f,g) \right\}.$$

\begin{cor}[\cite{ja2}, Corollary 5]\label{cT}
Assume that $X$ is a topological separable space and  $H\colon X \times X \to \R$ is a solution of 
\begin{equation}
\label{Ti}
H(f,h)\leq H(f,g)+ H(g,h), \quad f, g, h \in X 
\end{equation}
which is continuous  and equal to $0$ at every point of  $\Delta =\{ (f,f) : f \in X\}$.
Then 
\begin{equation}\label{repH}
H(f,g) = \sup \left\{ \f(f)-\f(g) : \f \in \mathcal{H}(H) \right\}, \quad f, g \in X.
\end{equation}
Conversely, for an arbitrary family $\mathcal{H}$ of real functions on $X$  every mapping $H\colon X \times X \to \R$ defined by \eqref{repH} solves \eqref{Ti}, it is equal to $0$ on $\Delta$ and  $\mathcal{H}\subseteq \mathcal{H}(H)$.
\end{cor}

On joining Proposition \ref{p2}, Proposition \ref{p3} and Corollary \ref{cT} we obtain the following corollary.

\begin{cor}\label{cH}
Assume that $X$ is a topological separable space, $S\colon X \times X \to \R$ and $G\colon X \times X \to \R$ satisfy inequality \eqref{add} for all $f, g, h \in X$, are continuous on the set $\Delta$ and $G(f,f)=0$ for every $f \in X$. Then
\begin{align}\nonumber
S(f,g) &= \sup \left\{ \f(f)-\f(g) : \f \in \mathcal{H}\(\frac12 (G+S)\) \right\} \\&- \sup \left\{ \g(f)-\g(g) : \g \in \mathcal{H}\(\frac12 (G-S)\) \right\}, \quad f, g \in X\nonumber
\end{align}
and
\begin{align}\nonumber
G(f,g) &= \sup \left\{ \f(f)-\f(g) : \f \in \mathcal{H}\(\frac12 (G+S)\) \right\} \\&+ \sup \left\{ \g(f)-\g(g) : \g \in \mathcal{H}\(\frac12 (G-S)\) \right\}, \quad f, g \in X.\nonumber
\end{align}

Conversely, for  arbitrary two families $\mathcal{H}_1, \mathcal{H}_2$ of real functions on $X$  every mappings $S\colon X \times X \to \R$, $G\colon X \times X \to \R$ defined by 
$$
S(f,g) = \sup \left\{ \f(f)-\f(g) : \f \in \mathcal{H}_1 \right\} - \sup \left\{ \g(f)-\g(g) : \g \in \mathcal{H}_2\right\}, \quad f, g \in X.
$$
and
$$
G(f,g) = \sup \left\{ \f(f)-\f(g) : \f \in \mathcal{H}_1 \right\} + \sup \left\{ \g(f)-\g(g) : \g \in \mathcal{H}_2\right\}, \quad f, g \in X.
$$
 solve \eqref{add}, both are equal to $0$ on $\Delta$, $\mathcal{H}_1\subseteq \mathcal{H}(\frac12 (G+S))$ and $\mathcal{H}_2\subseteq \mathcal{H}(\frac12 (G-S))$.
\end{cor}



\end{document}